\documentclass[11pt]{amsart}
\usepackage[centertags]{amsmath}
\usepackage{amsfonts,amssymb}
\usepackage{graphicx}
\usepackage{enumerate}
 \usepackage{url}

  \newtheorem{theorem}{Theorem}
  \newtheorem{corollary}{Corollary}

  \newtheorem{lemma}{Lemma}

  \DeclareMathOperator{\re}{Re}

\begin{document}

\title[Error terms when counting smooth and rough numbers]{A link between error terms when counting smooth and rough numbers}

\author{Andreas Weingartner} 
\address{Department of Mathematics, Southern Utah University, 351 West University Boulevard, Cedar City, Utah 84720, USA} 
\email{weingartner@suu.edu} 

\begin{abstract}
We establish a relationship between error terms appearing in estimates for the counting functions of smooth and rough numbers. 
We then apply this link to obtain an explicit upper bound for the error term  
in de Bruijn's approximation $\Lambda$ for the count of smooth numbers,
from an explicit upper bound, due to Fan, for the error term in a variant of de Bruijn's estimate for the count of rough numbers.
\end{abstract}

\maketitle

\section{Introduction}
Let $\mathcal{S}_y$ be the set of $y$-smooth (or $y$-friable) numbers, i.e. natural numbers all of whose prime factors are at most $y$, 
and define
$$
\Psi(x,y) := |\{1\le n \le x: p|n \Rightarrow p\le y\}| = |\mathcal{S}_y \cap [1,x]|,
$$
the corresponding counting function. Let 
$$
\Phi(x,y) := |\{1\le n\le x: p|n \Rightarrow p> y\}|,
$$
the number of $y$-rough (or $y$-sifted) integers up to $x$. With
$$
u:=\frac{\log x}{\log y},
$$
we have the classical estimates
$$
\Psi(x,y) \sim x \rho(u), \quad \Phi(x,y) \sim \frac{x \omega(u)}{\log y} \quad (u>1 \text{ fixed},\  x\to \infty),
$$
where $\rho(u)$ is Dickman's function and $\omega(u)$ is Buchstab's function.

De Bruijn \cite{BruI} introduced a more precise estimate for $\Psi(x,y)$, which can be defined as
(see \cite[Eq III.5.81]{Ten}) 
\begin{equation}\label{LambdaDef}
\Lambda(x,y) := x \rho(u)-\{x\}-x\int_0^u \rho'(u-v) \frac{\{y^v\}}{y^v} dv \qquad (x\ge 1, \ y\ge 2),
\end{equation}
where $\{x\}$ denotes the fractional part of $x$, and $\rho'(u)$ is defined by right-continutiy at $u=0,1$. 
De Bruijn \cite{Bru} also discovered an improved approximation for $\Phi(x,y)$, namely 
$$
W(x,y):= x \mu_y(u) \Pi(y) e^\gamma \log y,
$$
where $\gamma$ is Euler's constant,
$$
\mu_y(u) := \int_0^u \omega(u-v) y^{-v} dv,
$$
and
$$
\Pi(y) := \prod_{p\le y} \left(1-\frac{1}{p}\right) .
$$
We will work instead with the following two variants of $W(x,y)$:
$$
V(x,y):=1_{x\ge 1} + x \left( \Pi(y) - \frac{e^{-\gamma}}{\log y} +\mu_y(u)\right)
$$
and
$$
V^*(x,y):= 1_{x\ge 1} + x \mu_y(u).
$$
The function $V(x,y)$ appears naturally when approximating 
$\Phi(x,y)$ with the saddle point method (see \cite[Section III.6.4]{Ten}, where $V$ is called $W_1$), 
it is very close to $W(x,y)$ in a large domain (see \cite[Eq III.6.61]{Ten}), and it represents
the natural counterpart to $\Lambda(x,y)$ in Theorem \ref{thm1}.
The simpler expression $V^*(x,y)$ is close to both $V(x,y)$ and $W(x,y)$ when $y$ is large. 

We define the following error terms:
$$\Delta(x,y):= \Psi(x,y) -\Lambda(x,y),$$
$$
Q(x,y):= V(x,y) - \Phi(x,y), \quad Q^*(x,y):=V^*(x,y) - \Phi(x,y),
$$
$$
R(x,y):= \sum_{n\in S_y} Q(x/n,y), \quad R^*(x,y):= \sum_{n\in S_y} Q^*(x/n,y).
$$
With M\"{o}bius inversion (see Lemma \ref{LemMobInv}) we can express $Q$ in terms of $R$:
$$
Q(x,y)= \sum_{n\in S_y} \mu(n) R(x/n,y), \quad Q^*(x,y)= \sum_{n\in S_y} \mu(n) R^*(x/n,y).
$$
Our main result expresses $\Delta$ in terms of $R$ (resp. $R^*$) and vice versa.

\begin{theorem}\label{thm1}
For $x\ge 1$, $y\ge 2$, we have
\begin{equation}\label{thm1eq1}
\Delta(x,y) =R(x,y) + x \int_{-\infty}^u \frac{R(y^v,y)}{ y^{v}} \rho'(u-v) dv,
\end{equation}
\begin{equation}\label{thm1eq1*}
\Delta(x,y) =R^*(x,y) + x \int_{0}^u \frac{R^*(y^v,y)}{ y^{v}} \rho'(u-v) dv,
\end{equation}
\begin{equation}\label{thm1eq2}
R(x,y) = \Delta(x,y)+x \int_0^\infty \frac{\Delta(y^v,y)}{y^v} \left(\omega(u-v)-e^{-\gamma}\right)  dv,
\end{equation}
\begin{equation}\label{thm1eq2*}
R^*(x,y) = \Delta(x,y)+x \int_0^u \frac{\Delta(y^v,y)}{y^v}\omega(u-v)  dv.
\end{equation}

\end{theorem}

In the remainder of this section, we will state several corollaries to Theorem \ref{thm1} together with their short proofs.
The lemmas needed in those proofs are established in Section \ref{SecLemmas}.
We prove Theorem \ref{thm1} in Section \ref{SecProof}.

\begin{corollary}\label{corexact}
Let $X\ge 1$, $y\ge 2$ and $f(y)>0$. If 
$$
|Q(x,y)| \le x f(y) \quad (1\le x \le  X) ,
$$
then
$$
|R(x,y)| \le x f(y) \Pi(y)^{-1} \quad (0 < x \le X) 
$$
and 
$$
|\Delta(x,y)| \le 2 x f(y) \Pi(y)^{-1} \quad (1\le  x \le  X).
$$
This also holds if $(Q,R)$ is replaced by $(Q^*,R^*)$. 
\end{corollary}
\begin{proof}
The first implication follows from Lemma \ref{Lem0} and
\begin{equation}\label{Sum1}
 \sum_{n \in \mathcal{S}_y} \frac{1}{n}= \Pi(y)^{-1} .
\end{equation}
Equation \eqref{thm1eq1} (resp.  \eqref{thm1eq1*}) yields the bound for $|\Delta(x,y)| $.
\end{proof}

Corollary \ref{corexact} and the explicit estimates for $\Phi(x,y)$ due to Fan \cite{Fan} lead to the following explicit estimates for $\Delta(x,y)$.
\begin{corollary}\label{corFan}
We have
$$
|\Delta(x,y)| \le  \frac{15.8 x \log^{1/4} y}{\exp\{\sqrt{(\log y)/6.315}\} }  \qquad (x\ge 1, \ y\ge 2).
$$
Assuming the Riemann hypothesis (RH), we have
$$
|\Delta(x,y)| \le   \frac{1.66 x\log^2 y}{\sqrt{y}} \qquad (x\ge 1,\  y\ge 3).
$$
\end{corollary}
\begin{proof}
Corollary 1.2 of \cite{Fan} shows that $|\Phi(x,y)-x \mu_y(u)| \le x h(y)$, for $x\ge y \ge 2$,
where $h(y)$ is one of two explicit functions, depending on whether or not RH is assumed. 
This implies
$$|Q^*(x,y)|\le 1 + x h(y)  \le x(1/y + h(y) ) \quad (x\ge y\ge 2). $$ 
This upper bound is also valid if $1\le x<y$, since $Q^*(x,y) = 0$ in that case, by Lemma \ref{Lem0}.
Corollary \ref{corexact} 
with $f(y):= 1/y + h(y)$ yields
$$
|\Delta(x,y)| \le  2x(y^{-1} + h(y)) \Pi(y)^{-1} \quad (x\ge 1,\ y\ge 2).
$$
With the estimate $\Pi(y)^{-1}<e^\gamma \log y (1 + 0.5/\log^2 y)$ for $y\ge 286$ (see \cite[Thm 8]{RS}),
we obtain the stated bounds, for $y\ge 10^{18}$ without RH, and for $y\ge 30000$ with RH.
When $y$ is smaller than these values, the stated bounds are worse than the trivial bound $|\Delta(x,y)|\le x$:  Indeed, we have  
$0\le \Psi(x,y) \le \lfloor x \rfloor $ and $-\{x\} \le \Lambda(x,y) \le \lfloor x \rfloor $, which
follows at once from \eqref{LambdaDef}, since $0\le \{y^v\}/y^v \le 1$ and $\rho'(u)\le 0$.
\end{proof}

Corollary \ref{corexact2} gives an explicit bound for $Q^*$ based one for $\Delta$. 

\begin{corollary}\label{corexact2}
Let $X\ge 1$, $y\ge 2$ and $f(y)>0$. If 
$$
|\Delta(x,y)| \le x f(y) \rho(u) \quad (1\le x \le  X) ,
$$
then
$$
|R^*(x,y)| \le x f(y)  \quad (1\le  x \le X) 
$$
and 
$$
|Q^*(x,y)| \le x f(y) \Pi(y)^{-1} \quad (1\le x \le  X).
$$
\end{corollary}
\begin{proof}
The first implication is a consequence of \eqref{thm1eq2*} and the 
identity $\rho(u)+\int_0^u \rho(v) \omega(u-v) dv = 1$, which is \cite[Exercise 297]{Ten}.
The last estimate follows from Lemma \ref{Lem0}, Lemma \ref{LemMobInv} and \eqref{Sum1}.
\end{proof}

Corollary \ref{coriff} shows that upper bounds of the form $x f(y) e^{-Au}$   
translate nicely from $R$ to $\Delta$ and back. 

\begin{corollary}\label{coriff}
Let $A>0$ and $y_0\ge 2$ be fixed. We have   
$$
R(x,y)\ll x f(y) e^{-A u} \qquad (x\ge 1,\ y\ge y_0)
$$
if and only if 
$$
\Delta(x,y)\ll x f(y) e^{-A u} \qquad (x\ge 1, \ y\ge y_0).
$$
\end{corollary}
\begin{proof}
This follows from Theorem \ref{thm1}, Lemma \ref{Lem0} and the estimates \cite[Cor III.5.14 \& Cor III.6.5]{Ten}
$$
\rho'(u), \  \omega(u) -e^{-\gamma} \ll e^{-(A+1)u} \qquad (u\ge 0).
$$
\end{proof}

Define
$$
L_\varepsilon(y) := \exp\left\{ (\log  y)^{3/5 - \varepsilon}\right\}.
$$

\begin{corollary}\label{Cor4}
Let $\varepsilon$, $\delta$, $A$ be positive constants with $e^{A+\delta}\ge 2$.
We have
$$
Q(x,y), \ \Delta(x,y) \ll \frac{x e^{-Au}}{L_\varepsilon(y)} \qquad (x\ge 1,\ y\ge e^{A+\delta}).
$$
\end{corollary}

\begin{proof}
The upper bound for $Q(x,y) $ is Lemma \ref{PhiLem}, which is a consequence of \cite[Thm III.6.10]{Ten}. 
Applying this estimate with $\varepsilon$ replaced by 
$\varepsilon/2$, Lemma \ref{LemRQ} implies $R(x,y)\ll x e^{-Au}L_\varepsilon(y)^{-1} $. 
Corollary \ref{coriff} shows that the same estimate holds for $\Delta(x,y)$.
\end{proof}

De Bruijn \cite{BruI} established the bound $\Delta(x,y)\ll \frac{x u^2}{L_\varepsilon(y)}$ for $x\ge y\ge 2$. 
In the domain
\begin{equation}\label{Hdomain}
x\ge 3, \qquad y\ge 2, \qquad \exp\{ (\log\log x)^{5/3+\varepsilon}\} \le y \le x,
\end{equation}
we have the sharper estimates
\begin{equation}\label{QDrho}
Q(x,y), \ \Delta(x,y)\ll \frac{x \rho(u)}{L_\varepsilon(y)},
\end{equation}
due to Tenenbaum \cite[Thm III.6.10 \& Eq III.6.61]{Ten} for $Q$, and Saias \cite{Saias} for $\Delta$.
We note that the result for $\Delta$ in Corollary \ref{Cor4} can also be derived directly from \eqref{QDrho} 
in the domain \eqref{Hdomain}, and from \eqref{LambdaDef} and Lemma \ref{PsiLem} outside of \eqref{Hdomain}.

\section{Lemmas needed to derive the corollaries from Theorem \ref{thm1}}\label{SecLemmas}

\begin{lemma}\label{Lem0}
For $0< x \le y$, $y\ge 2$, we have
$$
Q^*(x,y)=0, \quad R^*(x,y)=0,
$$
$$
Q(x,y) = x\left( \Pi(y) - \frac{e^{-\gamma}}{\log y} \right), 
$$
$$
R(x,y) =  x\left( 1 - \frac{e^{-\gamma}}{\log y}\Pi(y)^{-1} \right) .
$$
\end{lemma}
\begin{proof}
The first three statements follow from $\Phi(x,y)=1_{x\ge 1} $ and $\mu_y(u)=0$ when $0<x\le y$.
The last claim now follows from  \eqref{Sum1}. 
\end{proof}

\begin{lemma}\label{LemMobInv}
For $x> 0$, $y\ge 2$, we have
\begin{equation*}
Q(x,y)= \sum_{n\in S_y} \mu(n) R(x/n,y), \quad Q^*(x,y)= \sum_{n\in S_y} \mu(n) R^*(x/n,y).
\end{equation*}
\end{lemma}
\begin{proof}
Writing $k=mn$, the first sum is
$$
\sum_{n\in S_y} \mu(n)\sum_{m\in S_y} Q(x/nm,y)
=\sum_{k\in S_y}  Q(x/k,y) \sum_{n|k} \mu(n) = Q(x,y).
$$
Changing the order of summation is justified because the double sum converges absolutely, by Lemma \ref{Lem0}.
The second claim follows likewise, with $Q^*$ replacing $Q$. 
\end{proof}

\begin{lemma}\label{PsiLem}
Let $A>0$ and $\delta>0$ be fixed. We have
$$
\Psi(x,y) \ll x e^{-Au} \qquad (x\ge 1, \ y\ge e^{A+\delta}).
$$
\end{lemma}
\begin{proof}
This follows from \cite[Thm III.5.2 \& Eq III.5.35]{Ten}, by considering several cases depending on the size of $y$ relative to $x$. 
\end{proof}

\begin{lemma}\label{LemRQ}
Let $A>0$ and $\delta>0$ be fixed. If
$$
Q(x,y) \ll x f(y) e^{-A u} \qquad (x\ge 1,\  y\ge e^{A+\delta})
$$
then
$$
R(x,y)  \ll x f(y) e^{-A u}  \log y \qquad (x\ge 1,\  y\ge e^{A+\delta}).
$$
 If
$$
R(x,y) \ll x f(y) e^{-A u} \qquad (x\ge 1,\  y\ge e^{A+\delta})
$$
then
$$
Q(x,y)  \ll x f(y) e^{-A u}  \log y \qquad (x\ge 1,\  y\ge e^{A+\delta}).
$$
\end{lemma}
\begin{proof}
Assume the first upper bound for $Q(x,y)$. When $0<x<1$, this bound still holds, 
by Lemma \ref{Lem0} and since it holds when $x=1$. Thus 
$$
\frac{R(x,y)}{x f(y)} \ll \sum_{n\in \mathcal{S}_y} \frac{ \exp(-A\log(x/n)/\log y)}{n}
= e^{-Au} \sum_{n\in \mathcal{S}_y} \frac{\exp(A\log(n)/\log y)}{n }.
$$
Abel summation and Lemma \ref{PsiLem}, with $(A+\delta/2,\delta/2)$ in place of $(A,\delta)$, 
shows that the last sum is $\ll \log y$, which implies the desired estimate for $R(x,y)$. 
The second part is proved similarly, with the help of Lemma \ref{LemMobInv}.
\end{proof}

\begin{lemma}\label{PhiLem}
Let $\varepsilon$, $\delta$, $A$ be positive constants with $e^{A+\delta}\ge 2$. We have
$$
Q(x,y)\ll \frac{x e^{-A u}}{L_\varepsilon(y)}\qquad (x\ge 1, \ y\ge e^{A+\delta}).
$$
\end{lemma}
\begin{proof}
If $1\le x\le y$, the result follows from Lemma \ref{Lem0} and
 an application of the prime number theorem.

If $x\ge y\ge e^{A+\delta}$, then the proof of \cite[Thm III.6.10]{Ten} shows that in the domain \eqref{Hdomain},
  $Q(x,y) \ll \Psi(x,y) L_\varepsilon(y)^{-1} \ll x e^{-Au}  L_\varepsilon(y)^{-1}$, by Lemma \ref{PsiLem}.

Outside of \eqref{Hdomain}, we have  $L_\varepsilon(y)\ll e^{u \delta/2}$. 
Dividing \cite[Eq III.6.52]{Ten} by $\Pi(y) e^\gamma \log y$ and combining the result 
with \cite[Thm III.6.2]{Ten} shows that in this case we have
$Q(x,y) \ll \Psi(x,y) \ll x e^{-(A+\delta/2)u} \ll  x e^{-Au}  L_\varepsilon(y)^{-1}$, by Lemma \ref{PsiLem},
with $(A+\delta/2,\delta/2)$ in place of $(A,\delta)$.
\end{proof}

\section{Proof of Theorem \ref{thm1}}\label{SecProof}

Let
$$
\alpha_y:=\frac{e^{-\gamma} }{\log y}\Pi(y)^{-1}, \qquad \beta_y  :=   \alpha_y -1.
$$
Note that 
\begin{equation}\label{RReq}
R^*(x,y)=R(x,y)+\beta_y x.
\end{equation}

\begin{lemma}\label{Lem1}
For $x\ge 0$, $y\ge 2$,
$$
[x] = \sum_{n \in \mathcal{S}_y} \Phi(x/n, y).
$$
\end{lemma}

\begin{proof}
Every natural $m\le x$ can be written uniquely as $m=nr$, where $n$ is $y$-smooth and $r$ is $y$-rough. 
\end{proof}

\begin{lemma}\label{LemFinalSum}
For $x\ge 1$, $y\ge 2$,
$$
\Psi(x,y) = \alpha_y x - x\sum_{n\in \mathcal{S}_y}\frac{1}{n } \mu_y\left(\frac{\log x/n}{\log y }\right)
+R(x,y)-\{x\}.
$$
For fixed $\varepsilon>0$, we have $R(x,y) \ll x e^{-u/2}L_\varepsilon(y)^{-1}$, for $x\ge 1$, $y\ge 2$.
\end{lemma}
\begin{proof}
The first statement follows from replacing $\Phi$ by $V-Q$ in Lemma \ref{Lem1}, the definitions of $V$ and $R$, and equation \eqref{Sum1}.
The second statement follows from Lemmas \ref{LemRQ} and \ref{PhiLem}, with $A=1/2$ and $\varepsilon$ replaced by $\varepsilon/2$.
\end{proof}

\begin{lemma}\label{LemInt}
For $x\ge 1$, $y\ge 2$,
$$
\Psi(x,y) =  \alpha_y x  - x\int_1^x \frac{\Psi(t,y)}{t^2 \log y}\omega\left(\frac{\log x/t}{\log y }\right) dt 
+R(x,y)-\{x\}.
$$
\end{lemma}
\begin{proof}
Integration by parts yields, for $u\ge 1$,
$$
\mu_y(u)\log y = \omega(u) - y^{1-u}-\int_0^{u-1} \omega'(u-v) y^{-v} dv=\omega(u)-\mu_y'(u).
$$
Thus, $\mu_y(u)\log y +\mu_y'(u) = \omega(u)$. This also holds for $u<1$ since all three terms vanish there. 
Partial summation shows that the sum in Lemma \ref{LemFinalSum} is the same as the integral in Lemma \ref{LemInt}.
\end{proof}

\begin{proof}[Proof of Theorem \ref{thm1}]
Write $\psi_y(u)=\Psi(y^u,y)/y^u$ and $t=y^v$. Lemma \ref{LemInt} implies, for $y\ge 2$ and $u\ge 0$,
$$
\psi_y(u)=\alpha_y - \int_0^u \psi_y(v) \omega(u-v) dv+ E_y(u),
$$
where 
$$
E_y(u) := R(y^u,y)/y^u - \{y^u\}/y^u .
$$
This leads to the equation of Laplace transforms, for $\re(s)>0$,
$$
\widehat{\psi}_y(s)=\frac{\alpha_y}{s} - \widehat{\psi}_y(s) \widehat{\omega}(s) + \widehat{E}_y(s).
$$
Since $s\widehat{\rho}(s) = 1/(1+\widehat{\omega}(s))$ by \cite[Thm III.6.7]{Ten}, we have
$$
\widehat{\psi}_y(s)=\alpha_y \widehat{\rho}(s) + \widehat{E}_y(s) s\widehat{\rho}(s)=\alpha_y \widehat{\rho}(s) + \widehat{E}_y(s) (1+\widehat{\rho'}(s)).
$$
Hence, for $u\ge 0$,
$$
\psi_y(u) = \alpha_y \rho(u) +E_y(u)+ \int_0^u E_y(v) \rho'(u-v) dv.
$$
Multiplying by $x$, we obtain, for $x\ge 1$, $y\ge 2$,
$$
\Psi(x,y) = \alpha_y x \rho(u)+R(x,y) - \{x\} + x\int_0^u \frac{R(y^v,y) -\{y^v\} }{y^v} \rho'(u-v) dv .
$$
Since  $ \alpha_y=\beta_y +1$, the definition of $\Lambda$ in \eqref{LambdaDef} implies
\begin{equation}\label{eqvar2}
\Delta(x,y)= \beta_y x \rho(u) +R(x,y) + x \int_{0}^u \frac{R(y^v,y)}{ y^{v}} \rho'(u-v) dv.
\end{equation}
The identity \eqref{thm1eq1} now follows from \eqref{eqvar2},
because $R(x,y)=-\beta_y x$ when $0<x<1$, by Lemma \ref{Lem0}.

Equation \eqref{thm1eq1*} follows from \eqref{eqvar2} and \eqref{RReq}. 

To derive \eqref{thm1eq2} from \eqref{eqvar2}, we define, for $y\ge 2$, $u\ge 0$,
$$
\delta_y(u):=\Delta(y^u,y)/y^u, \quad r_y(u):=R(y^u,y)/y^u.
$$ 
Dividing \eqref{eqvar2} by $x$, 
we obtain the equation of Laplace transforms
$$
\widehat{\delta}_y(s)=\beta_y \widehat{\rho}(s)+\widehat{r}_y(s)+ \widehat{r}_y(s)\widehat{\rho'}(s)
=\beta_y \widehat{\rho}(s)+ \widehat{r}_y(s)s \widehat{\rho}(s),
$$
for $\re(s)>0$.
Thus,
$$
\widehat{r}_y(s)=\frac{\widehat{\delta}_y(s)}{s \widehat{\rho}(s)}- \frac{\beta_y}{s}=\widehat{\delta}_y(s)(1+\widehat{\omega}(s)) - \frac{\beta_y}{s}.
$$
Inverting the Laplace transforms yields, for $x\ge 1$, $y\ge 2$,
\begin{equation}\label{eqR}
\frac{R(x,y)}{x}= \frac{\Delta(x,y)}{x} +  \int_0^u \frac{\Delta(y^v,y)}{y^v} \omega(u-v) dv - \beta_y .
\end{equation}
Letting $x\to \infty$ while keeping $y$ fixed, we find that
\begin{equation*}\label{intbeta}
 \int_0^\infty \frac{\Delta(y^v,y)}{y^v} e^{-\gamma} dv = \beta_y,
\end{equation*}
because
$$
R(x,y)\ll x e^{-u/2},\quad  \Delta(x,y)\ll x e^{-u/2},\quad \omega(u)-e^{-\gamma} \ll e^{-u},
$$
by Lemma \ref{LemFinalSum}, equation \eqref{thm1eq1} and \cite[Cor. III.6.5]{Ten}. This completes the proof of \eqref{thm1eq2}.

Equation \eqref{thm1eq2*} follows from \eqref{eqR} and \eqref{RReq}. 
\end{proof}

\section*{Acknowledgements}
The author thanks Eric Saias for a conversation that motivated this study.

\end{document}